\documentclass[a4paper]{amsart}
\usepackage{amsmath,amsthm,amssymb,amsfonts}
\usepackage{proof}
\usepackage[british]{babel}
\usepackage{pdfpages}
\usepackage[pdfborder={0 0 0}]{hyperref}
\usepackage{booktabs}
\usepackage{enumerate}
\usepackage{mathrsfs}
\usepackage{xypic}
\usepackage{stmaryrd}

\newtheorem{thm}{Theorem}[section]
\newtheorem*{thm*}{Theorem}
\newtheorem{cor}[thm]{Corollary}
\newtheorem{lem}[thm]{Lemma}
\newtheorem{prop}[thm]{Proposition}

\newtheorem{conj}[thm]{Conjecture}
\theoremstyle{definition}
\newtheorem{defi}[thm]{Definition}
\newtheorem{rem}[thm]{Remark}
\newtheorem{nota}[thm]{Notation}

\newtheorem{question}[thm]{Question}

\newcommand{\mathpzc}[1]{\mathcal{#1}}

\newcommand{\A}{\mathpzc{A}}

\newcommand{\B}{\mathpzc{B}}
\renewcommand{\phi}{\varphi}
\renewcommand{\epsilon}{\varepsilon}
\newcommand{\Th}{\mathrm{Th}}

\newcommand{\M}{\mathpzc{M}}

\newcommand{\C}{\mathpzc{C}}

\renewcommand{\P}{\mathcal{P}}

\newcommand{\ured}[1]{\leq_{#1}}
\newcommand{\nured}[1]{\not\leq_{#1}}
\newcommand{\uquiv}[1]{\equiv_{#1}}
\newcommand{\nuquiv}[1]{\not\equiv_{#1}}
\newcommand{\V}{\mathcal{V}}
\newcommand{\T}{\mathcal{T}}
\newcommand{\W}{\mathcal{W}}
\newcommand{\U}{\mathcal{U}}
\newcommand{\ugeq}[1]{\geq_{#1}}

\newcommand{\conc}{%
  \mathord{
    \mathchoice
    {\raisebox{1ex}{\scalebox{.7}{$\frown$}}}
    {\raisebox{1ex}{\scalebox{.7}{$\frown$}}}
    {\raisebox{.7ex}{\scalebox{.5}{$\frown$}}}
    {\raisebox{.7ex}{\scalebox{.5}{$\frown$}}}
  }
}

\newcommand{\Ss}{\mathcal{S}}

\begin{document}

\title{Levels of uniformity}

\author[R. Kuyper]{Rutger Kuyper}
\address[Rutger Kuyper]{Radboud University Nijmegen\\
Department of Mathematics\\
P.O. Box 9010, 6500 GL Nijmegen, the Netherlands.}
\email{mail@rutgerkuyper.com}
\thanks{The research of the second author was supported by NWO/DIAMANT grant 613.009.011 and by 
John Templeton Foundation grant 15619: `Mind, Mechanism and Mathematics: Turing Centenary Research Project'.}

\date{\today}

\begin{abstract}
\noindent 
We introduce a hierarchy of degree structures between the Medvedev and Muchnik lattices which allow varying amounts of non-uniformity. We use these structures to introduce the notion of the uniformity of a Muchnik reduction, which expresses how uniform a reduction is. We study this notion for several well-known reductions from algorithmic randomness. Furthermore, since our new structures are Brouwer algebras, we study their propositional theories. Finally, we study if our new structures are elementarily equivalent to each other.
\end{abstract}

\maketitle

\section{Introduction}

Over the years, the uses of the Medvedev and Muchnik lattices in computability theory have expanded far beyond their applications to intuitionistic logic, as originally intended by Medvedev. These two lattices formalise when a \emph{mass problem}, i.e.\ a set $\A \subseteq \omega^\omega$, is `easier' than another mass problem $\B$. Both say that a mass problem $\A$ is easier than a mass problem $\B$, or that $\A$ reduces to $\B$, if every function in $\B$ computes a function in $\A$. However, the Medvedev lattice imposes an additional restriction, saying that this reduction should be uniform, in the sense that the Turing machine performing the computation should be the same for every function in $\B$. On the other hand, in the Muchnik lattice we can choose a different machine for each function. Thus, the Medvedev and Muchnik lattices are the most uniform and the most non-uniform approach to reducing mass problems.

In practice, there are many reductions between mass problems that turn out to only be Muchnik reductions and not Medvedev reductions. This is especially true in algorithmic randomness, where the fact that the reductions are not uniform can often be shown using a straightforward majority vote argument. Thus, the conclusion is often that the reductions are only Muchnik reductions, i.e.\ that they are highly non-uniform. However, this is often not actually the case, but merely results from the fact that we do not yet know of a finer-grained hierarchy between Medvedev and Muchnik reducibility. This paper aims to resolve that problem by introducing exactly such a hierarchy.

In \cite{higuchi-kihara-2014}, Higuchi and Kihara introduced five structures between the Medvedev and Muchnik lattices. One of these is the lattice they call $\mathcal{D}^1_\omega$ and which we will call $\M_1$ in this paper. Roughly speaking, in this lattice the reductions do not have to be fully uniform, but the non-uniform choice has to be a $\Pi^0_1$-condition. In slightly more detail, a mass problem $\A$ $1$-reduces to a mass problem $\B$ if there is a uniformly $\Pi^0_1$-sequence $\V_0,\V_1,\dots$ covering $\B$ such that we can compute an element of $\A$ uniformly from an element $f \in \B$ \emph{together} with an $i \in \omega$ such that $f \in \V_i$.

A natural extension of this definition is to replace $\Pi^0_1$ by $\Pi^0_n$ in the informal definition just given. This is what we do in section \ref{sec-n-uni}, which yields the notion of $n$-reducibility, and the corresponding degree structures $\M_n$ for every $n \in \omega$. In that section, we will also show that the resulting degree structure is always a Brouwer algebra, a lattice with an additional implication operator which can be used to give semantics for propositional logics between intuitionistic logic and classical logic. The Medvedev and Muchnik lattices are also known to be Brouwer algebras. For $\M_1$ this was shown in \cite{higuchi-kihara-2014}, and it is the only Brouwer algebra among the intermediate degree structures studied by Higuchi and Kihara.

In section \ref{sec-maps} we study maps between $\M_n$ and $\M_m$ for $n \not= m$. We show that the natural surjection from $\M_n$ to $\M_m$ for $m > n$ preserves joins and meets, but not necessarily implications. On the other hand, we show that there are embeddings preserving joins and implications in the other direction.

Next, in section \ref{sec-uniformity} we introduce the \emph{uniformity} of a pair $(\A,\B)$ with $\A \leq_w \B$. This notion tries to capture how uniform a Muchnik reduction exactly is, as motivated above. The uniformity is the least number $n \in \omega$ such that $\A \leq_n \B$, if any such $n$ exists. Thus, the uniformity is the least $n \in \omega$ such that $\Pi^0_n$-choices suffice to make the reduction uniform.

We apply this notion of uniformity to algorithmic randomness in section \ref{sec-randomness}. Here we study the uniformity of some well-known Muchnik reductions from algorithmic randomness. This also allows us to give natural examples separating $n$-reducibility from $m$-reducibility for $n \not= m$.

As mentioned above, the structures $\M_n$ are all Brouwer algebras. We study their propositional theories as Brouwer algebras in sections \ref{sec-brouwer-1} and \ref{sec-brouwer-2}. While, just as for the Medvedev and Muchnik lattices, their theories are not intuitionistic propositional logic (IPC), we show that there are principal factors of $\M_n$ that do capture exactly IPC if $n \leq 1$ or $n \geq 4$ (a study motivated by Skvortsova's magnificent result stating that such factors exist for the Medvedev lattice). This also allows us to show that the theory of $\M_1$ is exactly Jankov's logic, the deductive closure of IPC and the weak law of the excluded middle $\neg p \vee \neg\neg p$, answering a question of Higuchi and Kihara \cite{higuchi-kihara-2014-2}. The problem remains open for $n=2$ and $n=3$.

Finally, in section \ref{sec-elementary} we study the first-order theories of the $\M_n$ as lattices. We show that $\M_n$ and $\M_m$ are not elementarily equivalent if $n \not = m$, except for the case $n,m \in \{0,1\}$ which we do not currently know how to deal with.

Our notation is mostly standard. We use $\oplus$ to denote joins or least upper bounds in lattices, and similarly $\otimes$ to denote meets or greatest lower bounds in lattices. When we write $\Psi(\tau)(n){\downarrow}$ for some Turing functional $\Psi: \omega^\omega \to \omega^\omega$, we mean that $\Psi$ halts on input $n$ in at most $|\tau|$ steps with the partial oracle $\tau$. We assume a fixed, computable pairing function $\langle n,m \rangle$. We let $\Phi_e$ be the Turing functional with index $e$. For any set $\A \subseteq \omega^\omega$, we let $C(\A)$ be the upwards closure of $\A$ under Turing reducibility, i.e.\ the set of those $f$ such that for some $g \in \A$ we have $f \geq_T g$. We fix effective listings $\{\Ss^n_e\}_{e \in \omega}$ of all $\Sigma^0_n$-classes and $\{\P^n_e = \overline{\Ss^n_e}\}_{e \in \omega}$ of all $\Pi^0_n$-classes. When we say that a sequence $\V_0,\V_1,\dots$ is uniformly $\Pi^0_n$, we mean there is a computable sequence $u_0,u_1,\dots$ such that $\V_i = \P^n_{u_i}$. We will denote concatenation both by $xy$ and by $x \conc y$.

For undefined notions from computability theory, we refer to Odifreddi \cite{odifreddi-1989}, for undefined notions from algorithmic randomness, we refer to Downey and Hirschfeldt \cite{downey-hirschfeldt-2010} and Nies \cite{nies-2008}, and for more background on the Medvedev lattice we refer to the surveys of Sorbi \cite{sorbi-1996} and Hinman \cite{hinman-2012}.

\section{The $n$-uniform degrees}\label{sec-n-uni}

In this section we will introduce the $n$-uniform degrees and prove some basic results about them.

\begin{defi}\label{defi-n-degrees}
Let $\A,\B \subseteq \omega^\omega$ and let $n \in \omega$. Then we say that $\A$ \emph{$n$-uniformly reduces to } $\B$ (notation: $\A \ured{n} \B$) if there exists a sequence $\V_0,\V_1,\dots$ of uniformly $\Pi^0_n$ sets with $\B \subseteq \bigcup_{i \in \omega} \V_i$ and a uniformly computable sequence $e_0,e_1,\dots$ such that for every $i \in \omega$ and every $f \in \B \cap \V_i$ we have $\Phi_{e_i}(f) \in \A$. If both $\A \ured{n} \B$ and $\B \ured{n} \A$ we say that $\A$ and $\B$ are \emph{$n$-uniformly equivalent} (notation: $\A \uquiv{n} \B$). We let $\M_n = \mathcal{P}\left(\omega^\omega\right) / {\uquiv{n}}$ and call its elements the \emph{$n$-uniform degrees}.
\end{defi}

We will often drop the adjective `uniform' and talk about $n$-reducibility, $n$-equivalence and the $n$-degrees instead. If $\A \leq_n \B$ and $\V_0,\V_1,\dots$ and $e_0,e_1,\dots$ are as in Definition \ref{defi-n-degrees} above, we say that these sequences \emph{witness} that $\A \leq_n \B$.

\begin{rem}\label{remark-sigma}
Note that, in Definition \ref{defi-n-degrees} we can replace $\Pi^0_n$ by $\Sigma^0_{n+1}$ without changing the concept. Namely, if there are a sequence $\U_i = \bigcup \V^i_j$ covering $\B$ with $\V_i$ uniformly $\Pi^0_n$ and a computable sequence $e_0,e_1,\dots$ such that $\Phi_{e_i}(f) \in \A$ for every $f \in \B \cap \V_i$, it is not hard to see that the sequences $\left( \V^i_j \right)_{i,j \in \omega}$ and $(e_i)_{i,j \in \omega}$ witness that $\A \leq_n \B$.
\end{rem}

Clearly the relation $\leq_n$ is reflexive on the $n$-degrees, but let us verify that it is also transitive.

\begin{prop}
$\ured{n}$ is transitive on $\mathcal{P}\left(\omega^\omega\right) \times \mathcal{P}\left(\omega^\omega\right)$, hence it induces an ordering on $\M_n$.
\end{prop}
\begin{proof}
Let $\A \ured{n} \B$ be witnessed by $\V_0,\V_1,\dots$ and $e_0,e_1,\dots$, and let $\B \ured{n} \C$ be witnessed by $\U_0,\U_1,\dots$ and $s_0,s_1,\dots$. For all $i,j \in \omega$, let $\W_{i,j}$ be as in Lemma \ref{lem-inverse-pi} below applied to $\V_i$ and $\Phi_{s_j}$. Now, if we let $\T_{j}$ be the class of functions $f$ for which $\Phi_{s_j}$ is total, then $\U_j \cap \C \subseteq \T_j$.
So,
\[\Phi_{s_j}(\W_{i,j} \cap \U_j \cap \C) \subseteq \V_i \cap \B.\]
Thus, $\Phi_{e_i}(\Phi_{s_j}(\W_{i,j} \cap \U_j \cap \C)) \subseteq \A$. On the other hand, the sequence $(\W_{i,j} \cap \U_j)_{i,j \in \omega}$ covers $\C$, which completes the proof.
\end{proof}

\begin{lem}\label{lem-inverse-pi}
Let $\Phi$ be a Turing functional and let $\T \subseteq \omega^\omega$ be the class of functions $f$ for which $\Phi(f)$ is total. Then for every $n \geq 1$ and every $\Pi^0_n$-class $\V$ we have that $\Phi^{-1}(\V)$ is $\Pi^0_n$ within $T$, i.e.\ there is a $\Pi^0_n$-class $\W$ such that $\Phi^{-1}(\V) = \W \cap T$. Furthermore, we can find an index for $\W$ uniformly in $n$ and indices for $\Phi$ and $\V$.
\end{lem}
\begin{proof}
Note that $\Phi^{-1}$ commutes with unions and intersections (for this we do not even need it to be computable), and that we have that $\Phi^{-1}(\llbracket\sigma\rrbracket)$ is $\Sigma^0_1$ uniformly in $\sigma$. In fact, $\Phi^{-1}(\llbracket\sigma\rrbracket)$ is also $\Pi^0_1$ within $\T$ uniformly in $\sigma$, i.e.\ there are $\W_\sigma$ which are $\Pi^0_1$ uniformly in $\sigma$ such that $\Phi^{-1}(\llbracket\sigma\rrbracket) = \W_\sigma \cap \T$. Indeed, let $\W_\sigma$ be the $\Pi^0_1$-class
\[\{f \mid \forall n(\Phi(f \restriction n) \restriction |\sigma|{\downarrow} \to \Phi(f \restriction n) \restriction |\sigma| = \sigma\}).\]
From this we can directly construct a $\W$ as required.
\end{proof}

Next, we show that $0$-reducibility is just Medvedev-reducibility, so the bottom level of our hierarchy of reducibilities is the completely uniform Medvedev reducibility.

\begin{prop}
Medvedev-reducibility and $0$-reducibility coincide.
\end{prop}
\begin{proof}
Clearly Medvedev-reducibility implies $0$-reducibility. Conversely, if $\A \leq_0 \B$, then by definition there are computable sequences $\sigma_0,\sigma_1,\dots$ and $e_0,e_1,\dots$ such that every $f \in \B$ starts with some string $\sigma_i$, and such that if $\sigma_i \subseteq f$ for $f \in \B$ then $\Phi_{e_i}(f) \in \A$. We can then uniformly compute an element of $\A$ from an element $f \in \B$ by computing the least $i$ such that $\sigma_i \subseteq f$ and sending $f$ to $\Phi_{e_i}(f)$.
\end{proof}

As mentioned in the introduction, for every $n$ the $n$-uniform degrees form a Brouwer algebra. This is what we prove next. For $n=0$, the Medvedev lattice, this was shown by Medvedev \cite{medvedev-1955}, and for the Muchnik lattice this was shown by Muchnik \cite{muchnik-1963}. 
For $n=1$ this result is due to Higuchi and Kihara \cite[Proposition 16]{higuchi-kihara-2014}, but the general proof below is our own.

Whenever we have a map $\alpha: \mathcal{P}\left(\omega^\omega\right) \to \mathcal{P}\left(\omega^\omega\right)$, we say that $\alpha$ \emph{induces} a map from $\M_n$ to $\M_m$ if, whenever $\A \equiv_n \B$, we have that $\alpha(\A) \equiv_m \alpha(\B)$. In this case we implicitly identify $\alpha$ with its induced map from $\M_n$ to $\M_m$ sending the $n$-degree of $\A$ to the $m$-degree of $\alpha(\A)$.

\begin{prop}
For every $n \in \omega$ the $n$-uniform degrees form a Brouwer algebra.
\end{prop}
\begin{proof}
The join and meet are induced by the same set-operation as in the Medvedev lattice, i.e.\ they are induced by the operations
\[\A \oplus \B = \left\{f \oplus g \mid f \in \A, g \in \B\right\}\]
and
\[\A \otimes \B = 0 \conc \A \cup 1 \conc \B\]
for $\A, \B \subseteq \omega^\omega$. The proofs are a straightforward generalisation of those for the Medvedev lattice, see e.g. Sorbi \cite[Theorem 1.3]{sorbi-1996}.

For the implication, consider the operation induced by
\begin{align*}
\A \to_n \B = \bigg\{u \conc e \conc f \mid &\left(\P^n_{\{u\}(i)}\right)_{i \in \omega} \text{ covers } \A \oplus \{f\}\\
&\text{and } \Phi_{\{e\}(i)}\left((\A \oplus \{f\}) \cap \P^n_{\{u\}(i)}\right) \subseteq \B\bigg\}.
\end{align*}
Here, $\left(\P^n_i\right)_{e \in \omega}$ is an effective enumeration of all $\Pi^0_n$-classes, as mentioned in the introduction.

Then $\A \oplus (\A \to_n \B) \ugeq{n} \B$: for this, consider the sequence $\V_0,\V_1,\dots$, where $\V_i$ consists of those $g \oplus (u \conc e \conc f)$ such that, if $\{u\}(i){\downarrow}$, then we have that $g \oplus f \in \P^n_{u(i)}$. Then each $\V_i$ is a $\Pi^0_n$-class, and if $g \oplus (u \conc e \conc f) \in (\A \oplus (\A \to_n \B)) \cap \V_i$ then $\Phi_{\{e\}(i)}(g \oplus f) \in \B$. Thus if we let $k_i$ be an index for the Turing functional sending $g \oplus (u \conc e \conc f)$ to $\Phi_{\{e\}(i)}(g \oplus f)$, then $\V_0,\V_1,\dots$ and $k_0,k_1,\dots$ witness that $\A \oplus (\A \to_n \B) \ugeq{n} \B$.

Conversely, if $\A \oplus \C \ugeq{n} \B$, fix $u$ and $e$ such that $\P^n_{u(0)},\P^n_{u(1)},\dots$ and $e(0),e(1),\dots$ witness this.
Then, for every $f \in \C$ we have $u \conc e \conc f \in \A \to_n \B$, which shows that $\A \oplus C \ugeq{n} \B$ if and only if $\C \ugeq{n} \A \to_n \B$.

We can now also easily show that $\to_n$ induces a well-defined operation on $\M_n$: if $\A_1 \uquiv{n} \A_2$ and $\B_1 \uquiv{n} \B_2$ we have that
\begin{align*}
\A_1 \to_n \B_1 \ured{n} \A_2 \to_n \B_2
&\Leftrightarrow \B_1 \ured{n} \A_1 \oplus (\A_2 \to_n \B_2)\\
&\Leftrightarrow \B_2 \ured{n} \A_2 \oplus (\A_2 \to_n \B_2)\\
&\Leftrightarrow \A_2 \to_n \B_2 \ured{n} \A_2 \to_n \B_2,
\end{align*}
where we use that we already know that $\oplus$ induces a well-defined operation on $\M_n$, as argued above. That $\A_2 \to_n \B_2 \ured{n} \A_1 \to_n \B_1$ follows in the same way.
\end{proof}

Note that Muchnik reducibility is implied by $n$-reducibility for every $n \in \omega$. To simplify the notation in the remainder of this paper we introduce the following notation.

\begin{nota}
Let $\ured{\omega}$, $\uquiv{\omega}$, $\M_\omega$ and $\to_\omega$ denote respectively Muchnik reducibility, Muchnik equivalence, the Muchnik lattice and the implication in the Muchnik lattice.
\end{nota}

Let us conclude this section by remarking that, instead of just looking at $n \in \omega$, we could also make a version of Definition \ref{defi-n-degrees} where we look at all ordinals $\alpha < \omega_1^\mathrm{CK}$. We expect many of the results in this paper hold in this more general setting, but also expect the proofs will get more technical. So, for reasons of clarity we have decided to restrict ourselves to $n \in \omega$, which we think already covers the most important part.

\section{Maps between $\M_n$ and $\M_m$}\label{sec-maps}

In this section we will show that there are natural maps between $\M_n$ and $\M_m$ for $n,m \in \omega + 1$. First we show that the natural surjection from $\M_n$ to $\M_m$ for $n < m$ preserves the lattice structure, but not the Brouwer algebra structure. This is known for the Medvedev and Muchnik lattices from Muchnik \cite{muchnik-1963}.

\begin{prop}
Let $0 \leq n \leq m \leq \omega$. Then the surjection from $\M_n$ onto $\M_m$ induced by the identity map on $\mathcal{P}(\omega^\omega)$ preserves $0, 1, \oplus$ and $\otimes$, but not necessarily $\to$.
\end{prop}
\begin{proof}
The preservation of $\oplus$ and $\otimes$ follows directly from the fact that they are induced by the same set-operations. That implication is not preserved follows from the fact proven below in Corollary \ref{cor-proper}
that for every $m$ there are $\A,\B$ such that $\A \ured{m} \B$ but $\A \nured{n} \B$ for all $n < m$; therefore $\B \to_m \A \uquiv{m} \omega^\omega$ while $\B \to_n \A \nuquiv{n} \omega^\omega$. Thus, $\B \to_n \A$ does not contain a computable element and is therefore not even Muchnik-equivalent to $\B \to_m \A$.
\end{proof}

It is known from Sorbi \cite{sorbi-1990} that there is an embedding of the Muchnik lattice into the Medvedev lattice preserving joins and implications.
Higuchi and Kihara \cite[Corollary 42]{higuchi-kihara-2014} showed that $\M_1$ also embeds into $\M$ as a partially ordered set. In fact, this embedding can be easily seen to preserve joins and implications, as we explain next.

\begin{thm}{\rm (\cite[Corollary 42]{higuchi-kihara-2014})}\label{thm-m1-into-m}
There is an embedding of $\M_1$ into $\M$ preserving joins and implications, induced by
\[\alpha(\A) = \{f \mid p(f) \in \A \text{ and } m(f) < \infty\},\]
where $m(f) = |\{i \mid f(i) = 0\}|$ and $p(f)$ is $\left(f \restriction [k,\infty)\right) - 1$ for the least $k$ such that $m(f \restriction [k,\infty)) = 0$.
Furthermore, $\alpha(\A) \uquiv{1} \A$ for all $\A$.
\end{thm}
\begin{proof}
We only prove that joins and implications are preserved, using the fact from the proof of \cite[Corollary 42]{higuchi-kihara-2014} that the map induced by $\alpha$ is well-defined, preserves the order and satisfies $\alpha(\A) \uquiv{1} \A$. Thus, we know that $\alpha(\A) \oplus \alpha(\B) \leq_\M \alpha(\A \oplus \B)$. Conversely, given $f \in \alpha(\A)$ and $g \in \alpha(\B)$, we show how to uniformly compute a function $h \in \alpha(\A \oplus \B)$. For this, we use an auxiliary number $k_s$, where we set $k_{-1} = 0$. At stage $s$, we define $h(2s)$ and $h(2s+1)$. First, check if either $f(s) = 0$ or $g(s) = 0$. If so, let $h(2s) = h(2s+1) = 0$ and let $k_{s} = 0$. Otherwise, let $h(2s) = f(k_{s-1})$, $h(2s+1) = g(k_{s-1})$ and $k_s = k_{s-1} + 1$. Then it can be directly verified that $h \in \alpha(\A \oplus \B)$, and since we computed $h$ uniformly in $f$ and $g$ we have $\alpha(\A) \oplus \alpha(\B) \equiv_\M \alpha(\A \oplus \B)$.

For the implication, again we already know that $\alpha(\A) \to \alpha(\B) \leq_\M \alpha(\A \to \B)$ from the fact that the order is preserved. Conversely, we have that $\alpha(\A) \to \alpha(\B) \uquiv{1} \A \to \B$ because $\alpha(\A) \uquiv{1} \A$ and $\alpha(\B) \uquiv{1} \B$. Thus, $\alpha(\alpha(\A) \to \alpha(\B)) \equiv_\M \alpha(\A \to \B)$. However, $\C \geq_\M \alpha(\C)$ holds for any $\C$ by sending $f$ to $f+1$, so we see that $\alpha(\A) \to \alpha(\B) \geq_\M \alpha(\A \to \B)$, as desired.
\end{proof}

We now show that we have such embeddings for all $n,m \in \omega + 1$ with $m \leq n$.

\begin{thm}\label{thm-embed}
Let $n \in \omega+1$. Then there exists a map $u_n: \mathcal{P}(\omega^\omega) \to \mathcal{P}(\omega^\omega)$ such that for every $m \leq n$, the map $u_{n,m}: \M_n \to \M_m$ induced by $u_n$ is a well-defined embedding preserving joins and implications.
Furthermore, $u_n(\A) \uquiv{n} \A$ for all $n$ and $\A$.
\end{thm}
\begin{proof}
If $n = \omega$, we can use Sorbi's embedding of the Muchnik into the Medvedev degrees mentioned above. Also, if $n = 0$ there is nothing to be proven, and if $n = 1$ this follows from Theorem \ref{thm-m1-into-m}. So, we let $n \in \omega$ with $n \geq 2$.

In the case of the Muchnik lattice, we have a natural representative of the Muchnik degree of $\A$ in the Medvedev lattice: 
take the Medvedev degree of the upwards closure of $\A$, which is the same as the Medvedev degree of $\bigcup \{\B \subseteq \omega^\omega \mid \B \geq_w \A\}$. In other words, the Muchnik degree of $\A$ contains a maximal mass problem. This is what is used in the embedding of Sorbi for the Muchnik lattice. In our current case it is harder to find a natural representative; we cannot just take $\bigcup \{\B \subseteq \omega^\omega \mid \B \geq_n \A\}$ because $\A$ does not $n$-reduce to this set in general. This is caused by the fact that in general there is not a `universal' sequence of uniform $\Pi^0_n$-classes covering $\A$.

Thus, we need to find a different approach which works around this non-existence of a universal sequence. For this, note that, given any set $\C \subseteq \mathcal{P}\left(2^{\omega^{n}}\right)$ there is a natural $\Pi^0_n$-class $\V$ which covers part of $\C$: take those $X$ such that
\[\forall m_1 \exists m_2 \dots \forall m_{n-1} \exists m_n ((m_1,\dots,m_n) \in X).\]
We will show that this is in a certain sense a universal way of making $\Pi^0_n$-choices. Furthermore, if we do not take $\omega^n$ but $\omega^{n+1}$, by a slight modification we can make a natural $\Pi^0_n$-class which uses this extra space to code functions, and if we go up to $\omega^{n+2}$ we can use this extra space to be able to deal with multiple $\Pi^0_n$-classes at once. We will use this to define our desired representative.

\bigskip
We now give the full details. Consider $f \in (\omega^{n+1} \times \omega)^\omega$ (where we implicitly identify $(\omega^{n+1} \times \omega)^\omega$ with $\omega^\omega$ in some computable way). In what follows, when we write $x \in f$ we means that there is some $i \in \omega$ with $f(i) = x$. We now inductively define when \emph{$f$ is $(\rho,s)\textrm{-}\Sigma$-valid} and when \emph{$f$ is $(\rho,s)\textrm{-}\Pi$-valid}, where $1 \leq s \leq n$ and $\rho \in \omega^{n+1-s}$:
\begin{itemize}
\item $f$ is $(\rho,1)\textrm{-}\Pi$-valid if and only if there is a $k \in \omega$ such that for all $m \in \omega$, $(\rho \conc m,k) \in f$.
\item $f$ is $(\rho,1)\textrm{-}\Sigma$-valid if and only if there are a $k \in \omega$ and an $m \in \omega$ such that $(\rho \conc m,k) \in f$.
\item $f$ is $(\rho,s+1)\textrm{-}\Pi$-valid if and only if it is $(\rho \conc m,s)\textrm{-}\Sigma$-valid for every $m \in \omega$.
\item $f$ is $(\rho,s+1)\textrm{-}\Sigma$-valid if and only if it is $(\rho \conc m,s)\textrm{-}\Pi$-valid for some $m \in \omega$.
\end{itemize}

First, let uw consider how valid $f$ code functions in a computable way. If $f$ is $(\rho,3)\textrm{-}\Pi$-valid, we know that for each $m \in \omega$ there is an $a_m$ such that $f$ is $(\rho \conc m \conc a_m,1)\textrm{-}\Pi$-valid, and therefore there is a $k_m$ such that $(\rho \conc m \conc a_m \conc t,k_m) \in f$ for every $t \in \omega$. Given $m$, let $k_m$ be least (in the order $f(0),f(1),\dots$) such that $(\rho \conc m \conc a_m \conc t,k_m) \in f$ for some $a_m,t \in \omega$. In this case we let $p_\textrm{odd}(\rho,f) = k_0 k_1 \dots$. This is computable uniformly in $\rho$ and $f$.

On the other hand, if $f$ is $(\rho,2)\textrm{-}\Pi$-valid, we know that for each $m \in \omega$ there are $k_m$ and $a_m$ such that $(\rho \conc m \conc a_m, k_m) \in f$. Again let $k_m$ be the least such $k_m$ and define $p_\textrm{even}(\rho,f) = k_0 k_1 \dots$, then this is computable uniformly in $\rho$ and $f$.

We now let
\begin{align*}
u_n(\A) = \{f \mid &\exists i. f \text{ is } (i,n)\textrm{-}\Pi\text{-valid, and }\\
&\forall \rho.f \text{ is } (\rho,3)\textrm{-}\Pi\text{-valid and } f \text{ is } (\rho \restriction 1,n)\textrm{-}\Pi\text{-valid} \Rightarrow p_\textrm{odd}(\rho,f) \in \A\}
\end{align*}
if $n$ is odd, and
\begin{align*}
u_n(\A) = \{f \mid &\exists i. f \text{ is } (i,n)\textrm{-}\Pi\text{-valid, and }\\
&\forall \rho.f \text{ is } (\rho,2)\textrm{-}\Pi\text{-valid and } f \text{ is } (\rho \restriction 1,n)\textrm{-}\Pi\text{-valid} \Rightarrow p_\textrm{even}(\rho,f) \in \A\}
\end{align*}
if $n$ is even.
For the remainder of the proof let us assume that $n$ is odd; the even case proceeds in the same way.

\bigskip

First, we claim: $u_n(\A) \uquiv{n} \A$. Indeed, to show that $u_n(\A) \leq_n \A$, or in fact even $u_n(\A) \leq_\M \A$, send $f$ to the function $g$ given by
\[g(\langle i,a_1,\dots,a_n\rangle) = (ia_1 \dots a_n,f(a_{n-2})).\]
It is easy to verify that this gives an element of $u_n(\A)$.

Conversely, consider the classes $\V_{\rho}$
consisting of those $f$ such that $f$ is $(\rho \restriction 1,n)\textrm{-}\Pi$-valid and such that f is $(\rho,3)\textrm{-}\Pi$-valid. The first is a $\Pi^0_n$-condition, and the latter is $\Pi^0_3$. Thus this is a $\Pi^0_n$-class uniformly in $\rho$ since $n \geq 2$ and $n$ is odd, hence $n \geq 3$. Also note that it covers $u_n(\A)$ because if $f$ is $(i,n)\textrm{-}\Pi$-valid, it is $(\rho,3)\textrm{-}\Pi$-valid for some string $\rho$ by definition.
Furthermore, for each such $f$ we can uniformly compute an element of $\A$ given a $\rho$ such that $f \in \V_\rho$, by computing $p_\textrm{odd}(\rho,f)$. This shows that $\A \leq_n u_n(\A)$ and hence $u_n(\A) \equiv_n \A$.

In particular, if $m \leq n$ we see that $u_n(\A) \leq_m u_n(\B)$ implies $\A \ured{n} \B$.
Next, assume $\A \ured{n} \B$. Then also $\A \ured{n} u_n(\B)$, as shown above. We will show that $u_n(\A) \leq_\M u_n(\B)$. Fix a computable sequence $\sigma_{i,a_1,\dots,a_n}$ and a computable sequence $e_0,e_1,\dots$ such that $\Phi_{e_i}(\V_i \cap u_n(\B)) \subseteq \A$, where
\[\V_i = \bigcap_{a_1 \in \omega} \bigcup_{a_2 \in \omega} \dots \bigcap_{a_n \in \omega} \llbracket \sigma_{i,a_1,\dots,a_n}\rrbracket.\]
Given any $f \in u_n(\B)$, we show how to uniformly compute an element $\Psi(f)$ of $u_n(\A)$. To define $\Phi(f)(m)$, wait until a stage $s$ such that $\sigma_{i,a_1,\dots,a_n} \subseteq f$ for some $i,a_1,\dots,a_n \leq s$ for which no element of $\Phi(f) \restriction m$ begins with $i,a_1,\dots,a_n$ and such that $\Phi_{e_i}(f)(a_{n-2})[s]{\downarrow}$. If so, let $(i,a_1,\dots,a_n)$ be the least sequence (in some fixed computable well-ordering of $\omega^{n+1}$) for which this holds, and let $f(m) = (i a_1 \dots a_n,\Phi_{e_i}(f)(a_{n-2}))$. Then $\Psi(f)$ is total because $f$ is in $u_n(\B) \cap \V_i$ for some $i \in \omega$. We can also directly verify that $\Psi(f)$ is $(i,n)\textrm{-}\Pi$-valid for this $i$.

Finally, note that for every $m \in \omega$, letting $\Psi(f)(m) = (ia_1\dots a_n,b)$, we have that $b = \Phi_{e_i}(f)(a_{n-2})$ by construction. Furthermore, $\Psi(f)$ is $(i,n)\textrm{-}\Pi$-valid if and only if
\[\forall a_1 \exists a_2 \dots \forall a_n (\sigma_{i,a_1,\dots,a_n} \subseteq f),\]
if and only if $f \in \V_i$.
From this we can directly verify that, if $\Psi(f)$ is $(\rho,3)\textrm{-}\Pi$-valid and $(\rho \restriction 1,n)\textrm{-}\Pi$-valid, we have that $p_\textrm{odd}(\rho,\Psi(f)) = \Phi_{e_{\rho(0)}}(f) \in \A$. Thus, $\Psi(f) \in u_n(\A)$ and therefore $\A \ured{n} \B$ if and only if $u_n(\A) \leq_\M u_n(\B)$. This also implies that the induced map is well-defined for every $m$.

\bigskip

We have already seen that $u_n(\A) \oplus u_n(\B) \leq_\M u_n(\A \oplus \B)$. Conversely, given $f \in u_n(\A)$ and $g \in u_n(\B)$, we show how to construct a function $\Psi(f \oplus g) \in u_n(\A \oplus \B)$. What we basically need to use is that the conjunction of two $\Sigma^0_{n+1}$-formulas is again a $\Sigma^0_{n+1}$-formula. However, we need to make sure we preserve the coding, which takes a little more work.

To define $\Psi(f \oplus g)(\langle 2i,j,a_2,a_4,\dots,a_{n-4}\rangle)$, we let $f(i) = (s b_1 \dots b_n,x)$. Now we define
\begin{align*}\Psi(f \oplus g)&(\langle 2i,j,a_2,a_4,\dots,a_{n-4}\rangle)\\
=\; &(\langle s,j\rangle \conc b_1 \conc \langle b_2,a_2\rangle \conc b_3 \cdots \conc b_{n-3} \conc (2 \cdot b_{n-2}) \conc b_{n-1} \conc b_n,x).
\end{align*}
To define $\Psi(f \oplus g)(\langle 2i+1,a_2,a_4,\dots,a_{n-4}\rangle)$, we let $g(i) = (s b_1 \dots b_n,x)$. Now we define
\begin{align*}
\Psi(f \oplus g)&(\langle 2i+1,j,a_2,a_4,\dots,a_{n-4}\rangle)\\
=\; &(\langle j,s\rangle \conc b_1 \conc \langle a_2,b_2\rangle \conc b_3 \cdots \conc b_{n-3} \conc (2 \cdot b_{n-2}+1) \conc b_{n-1} \conc b_n,x).
\end{align*}

Then, if $\rho = (\langle s,t \rangle \conc a_1 \conc \langle a_2,b_2\rangle \conc \cdots \conc a_{n-3})$, it can be directly verified that $\Psi(f \oplus g)$ is $(\rho,3)\textrm{-}\Pi$-valid if and only if $f$ is $(s a_1 a_2 \dots a_{n-3},3)\textrm{-}\Pi$-valid and $g$ is $(t a_1 b_2 \dots a_{n-3},3)\textrm{-}\Pi$-valid.
Furthermore, in this case we have that
\[p_\textrm{odd}\left(\rho,\Psi(f \oplus g)\right) = p_\textrm{odd}(\rho,f) \oplus p_\textrm{odd}(\rho,g).\]
Now, a straightforward calculation now shows that $\Psi(f \oplus g)$ is $(\langle s,t \rangle,n)\textrm{-}\Pi$-valid if and only if $f$ is $(s,n)\textrm{-}\Pi$-valid and $g$ is $(t,n)\textrm{-}\Pi$-valid. Combining all of this we see that $\Psi(f \oplus g) \in u_n(\A \oplus \B)$, which is what we needed to show.

\bigskip

Finally, that the implication is preserved can be proven in the same way as in the proof of Theorem \ref{thm-m1-into-m}, using that $\C \geq_\M u_n(\C)$ for every $\C$ as shown above.
\end{proof}
%meets are not preserved, since Muchnik degrees are mapped to Muchnik degrees and we have a result about this below

\section{Uniformity}\label{sec-uniformity}
Using the $n$-degrees, we can now introduce the measure of uniformity mentioned in the introduction.

\begin{defi}
Let $\A \leq_w \B$. Then we say that the \emph{uniformity of $\A$ to $\B$} is the least $n \in \omega+1$ such that $\A \ured{n} \B$.
\end{defi}

As we will see later, there are cases when the uniformity is not a natural number. However, if $\A$ is a reasonable class, in the sense that it is arithmetical, then it turns out that the uniformity is in fact a natural number. This follows from the following result, which is due to Higuchi and Kihara \cite{higuchi-kihara-2014}, but the proof given here is our own.

\begin{prop}{\rm (\cite[Proposition 27]{higuchi-kihara-2014})}\label{prop-max-uni}
Let $\A \leq_w \B$ be such that $\A$ is $\Sigma^0_{n+1}$. Then the uniformity of $\A$ to $\B$ is at most $\max(n,2)$.
\end{prop}
\begin{proof}
Let $\Phi$ be a Turing functional.
Let $\T \subseteq \omega^\omega$ be the class of functions $f$ for which $\Phi(f)$ is total, which is a $\Pi^0_2$-class. From Lemma \ref{lem-inverse-pi} we then see that we have that $\Phi^{-1}(\A) \cap \T$ is $\Sigma^0_{\max(n,2)}$.

Now, let $\V_i$ be $\Phi_i^{-1}(\A) \cap \T$, and let $e_i = i$. Then the $\V_i$ cover $\B$ since we assumed that $\A \leq_w \B$, and $\Phi_e(\V_e) \subseteq \A$. The result now follows, as discussed in Remark \ref{remark-sigma}.
\end{proof}

\begin{rem}
Note that, in the proof above, the reduction does not depend on $\B$, but only on $\A$. Thus, given any $\Sigma^0_{n+1}$-class $\A$ there is a single reduction which witnesses that $\A \leq_{\max(n,2)} \B$ for every $\B$ with $\B \geq_w \A$.
\end{rem}

\begin{cor}
If $\A$ is arithmetical, then the uniformity of $\A$ to $\B$ is a natural number.
\end{cor}

In Theorem \ref{thm-dnc} below we will see that Proposition \ref{prop-max-uni} is optimal for $n \geq 2$. For $n=2$ this follows from the following elegant pair of theorems. 

\begin{thm}{\rm (Jockusch \cite{jockusch-1989})}
We have that $\mathrm{DNR}_2 \equiv_w \mathrm{DNR}_3$, but $\mathrm{DNR}_2 \not\leq_\M \mathrm{DNR}_3$.
\end{thm}

\begin{thm}{\rm (Higuchi and Kihara \cite[Corollary 72]{higuchi-kihara-2014-2})}
\[\mathrm{DNR}_2 \not\leq_1 \mathrm{DNR}_3.\]
\end{thm}

\begin{cor}
The uniformity of $\mathrm{DNR}_2$ to $\mathrm{DNR}_3$ is 2.
\end{cor}
\begin{proof}
From the previous two theorems and the fact that $\mathrm{DNR}_2$ is a $\Pi^0_1$-class.
\end{proof}

\section{Uniformity and algorithmic randomness}\label{sec-randomness}

To illustrate the definition given in the previous section, we will now study the uniformity of some well-known Muchnik reductions from algorithmic randomness. First, we will study a version of the effective 0-1-law, which is originally due to Ku\v{c}era \cite{kucera-1985}. This will be a helpful tool during the remainder of this paper.

\begin{thm}{\rm (Effective 0-1-law, Ku\v{c}era \cite{kucera-1985}, Kautz \cite{kautz-1991})}\label{eff-0-1}
Let $n \in \omega$, let $\V$ be a $\Pi^0_n$-class of positive measure and let $X$ be $n$-random. Then there is a $k \in \omega$ with $X \restriction [k,\infty) \in \V$.
\end{thm}
\begin{proof}
See e.g.\ Downey and Hirschfeldt \cite[Theorem 6.10.2]{downey-hirschfeldt-2010}.
\end{proof}

\begin{thm}\label{eff-0-1-uni}
Let $n \in \omega$, let $\A$ be a mass problem and let $n\textrm{-Random}$ be the class of $n$-randoms. Then $\A \ured{n} n\textrm{-Random}$ if and only if there exists a $\Pi^0_n$-class $\V$ of positive measure such that $\A \leq_\M \V$.
\end{thm}
\begin{proof}
First, assume $\A \ured{n} n\textrm{-Random}$ and let this be witnessed by $\V_0,\V_1,\dots$ and $e_0,e_1,\dots$. Then some $\V_i$ has positive measure by countable additivity, and $\A \leq_\M \V_i$ as witnessed by $e_i$.

Conversely, let $\V$ be a $\Pi^0_n$-class of positive measure and let $\Psi$ be such that $\Psi(\V) \subseteq \A$.
By the effective 0-1-law, we know that for each $n$-random $X$ there is an $m \in \omega$ such that $X \restriction [m,\infty) \in \V$.
Let $\V_m = 2^{< m} \conc \V$, i.e. $\V_m$ is the $\Pi^0_n$-class consisting of those $X$ such that $X \restriction [m,\infty) \in \V$.
Then the $\V_m$ cover the class of $n$-randoms, as we have just argued. Furthermore, if $m$ is such that $X \in \V_m$, we can compute an element of $\A$ uniformly from $X$ and $m$, namely $\Psi(X \restriction [m,\infty))$.
\end{proof}

That for some $\Pi^0_n$-classes $\V$ we have that $n$-reducibility is optimal, i.e.\ that there are $\Pi^0_n$-classes $\V$ and mass problems $\A$ such that $\A \leq_\M \V$ but $\A \not\leq_m n\textrm{-Random}$ for any $m < n$, will follow from the next theorem.

For $n \in \omega$ with $n \geq 1$, we say that $f$ is \emph{$n$-DNC} if $f$ is DNC relative to $\emptyset^{(n-1)}$. Furthermore, we say that $X$ is \emph{$\omega$-random} if $X$ is arithmetically random, i.e.\ if $X$ is $n$-random for every $n \in \omega$, and we say that $f$ is \emph{$\omega$-DNC} if $f$ is $n$-DNC for every $n \in \omega$.  Ku\v{c}era \cite{kucera-1985} has shown that every $n$-random set computes an $n$-DNC function, which lies at the basis of the following theorem.

\begin{thm}\label{thm-dnc}
Let $n \in \omega + 1$ with $n \geq 1$. Then $n\textrm{-DNC}$ Muchnik-reduces to $n$-randomness, with uniformity $n$.
\end{thm}
\begin{proof}
Let us first consider $n \in \omega$.
From the standard proof of Ku\v{c}era's result mentioned above (see e.g.\ \cite[Theorem 8.8.1]{downey-hirschfeldt-2010}) we know that there is a $\Pi^{0,\emptyset^{(n-1)}}_1$-class $\V$ (which is a specific kind of a $\Pi^0_n$-class) for which $n$-DNC Medvedev-reduces to $\V$. Now apply Theorem \ref{eff-0-1-uni}. 
Alternatively, for $n \geq 2$ we can use that the collection of $n$-DNC functions is $\Pi^{0,\emptyset^{(n-1)}}_1$, and apply Proposition \ref{prop-max-uni}.

On the other hand, let $m < n$ and assume towards a contradiction that $n$-DNC $m$-reduces to $n$-randomness. Fix witnesses $\V_0,\V_1,\dots$ and $e_0,e_1,\dots$ for this fact (so $\V_0,\V_1,\dots$ are uniformly $\Pi^0_m$). Without loss of generality, we may assume that every $\V_i$ is a subclass of $2^\omega$.
We show that $\emptyset^{(m)}$ computes an $n$-DNC function $f$, a contradiction.

Indeed, given $k \in \omega$, to define $f(k)$, look for the first $\sigma \in 2^{<\omega}$ and $i \in \omega$ such that $\V_i$ has strictly positive measure above $\sigma$ and such that $\Phi_{e_i}(\sigma)(k)\downarrow$. Since the measure of a $\Pi^0_m$-class is $\emptyset^{(m)}$-computable, we can do this computably in $\emptyset^{(m)}$. Furthermore, note that such $\sigma$ and $i$ have to exist by countable additivity. Now $\Phi_{e_i}(\sigma)(k) \not= \{k\}^{\emptyset^{(n-1)}}(k)$ since $\sigma$ can be extended to an n-random within $\V_i$ (after all, it has positive measure above $\sigma)$ and $\Phi_{e_i}(\V_i) \subseteq n\textrm{-DNC}$. Thus, $f$ is $n$-DNC, as desired.

Finally, let us consider $n=\omega$. The upper bound (i.e.\ the fact that $\omega$-DNC Muchnik-reduces to $\omega$-randomness) again follows from the standard proof. On the other hand, if $\omega$-randomness $m$-reduces to $\omega$-DNC for some natural number $m$, then the same argument as above shows that $\emptyset^{(m)}$ computes an $\omega$-DNC function, which is again a contradiction.
\end{proof}

The previous theorem now also allows us to separate $n$-reducibility from $m$-reducibility for $n \not= m$.

\begin{cor}\label{cor-proper}
For every $n \in \omega + 1$ there are mass problems $\A \leq_n \B$ such that for no $m < n$ we have $\A \leq_m \B$.
\end{cor}

Recall that $n$-DNC$_{2^m}$ is the class of $n$-DNC functions $f$ for which $f(m) \leq 2^m$.

\begin{cor}
Let $n \in \omega + 1$ with $n \geq 1$. Then $n$-DNC$_{2^m}$ Muchnik-reduces to $n$-randomness, with uniformity $n$. \end{cor}
\begin{proof}
This follows from the same proof as Theorem \ref{thm-dnc}.
\end{proof}

Next, let us compare our notion with the notion of layerwise computability introduced by Hoyrup and Rojas \cite{hoyrup-rojas-2009}. When looking at the uniformity of some $\A$ to $n$-randomness, we allow arbitrary $\Pi^0_n$-classes covering the class of $n$-randoms to witness this reduction. Instead, we could also decide to only allow the class $\V_n$ to be the $n$th layer, i.e.\ the complement of $\U_n$ for some fixed universal $n$-randomness test $\U_0,\U_1,\dots$. We now show that this is strictly weaker than our notion.

\begin{prop}
We do not have that $n$-DNC$_{2^m}$ reduces layerwise to $n$-ran\-dom\-ness; that is, there is no computable 
$e_0,e_1,\dots$ such that $\Phi_{e_i}(\V_i)$ is contained in $n$-DNC$_{2^m}$ for every $i \in \omega$, where $\V_i$ is the complement of $\U_i$ and $\U_0,\U_1,\dots$ is some fixed universal $n$-randomness test.
\end{prop}
\begin{proof}
Towards a contradiction, assume such a sequence $e_i$ exists; we show that there is a computable $n$-DNC$_{2^m}$ function $f$, which is a contradiction. Given $m$, to define $f(m)$, look for the first $i \in \omega$ and $k \in \omega$ such that at least measure $2^{-i+1}$-many strings $\sigma$ satisfy $\Phi_{e_i}(\sigma)(m){\downarrow}=k$. Such a string must exist, since for every $X \in \V_i$ we have that $\Phi_{e_i}(X)(m){\downarrow} \in \{0,\dots,2^m\}$ and $\V_i$ has measure at least $1 - 2^{-i}$, so at least measure $\frac{1 - 2^{-i}}{2^m + 1}$-many $X$ must be sent to the same value $k$, and if $i$ is large enough then $\frac{1 - 2^{-i}}{2^m + 1}$ is at least $2^{-i+1}$.
Furthermore, for any $i$ and $k$ such that at least measure $2^{-i+1}$-many strings $\sigma$ satisfy $\Phi_{e_i}(\sigma)(m){\downarrow}=k$ we know that
\[\mu(\V_i \cap \{X \mid \Phi_{e_i}(X)(m){\downarrow}=k\}) \geq 1 - (2^{-i} + 1 - 2^{-i+1}) = 2^{-i} > 0,\]
so there is some $X$ in this set and for such $X$ we have that $\Phi_{e_i}(X)$ is $n$-DNC$_{2^m}$, so $k \not= \{m\}^{\emptyset^{(n-1)}}(m)$. Therefore $f$ is $n$-DNC$_{2^m}$, as desired.
\end{proof}

\begin{rem}
If we formulate the previous proposition with $n$-DNC instead of $n$-DNC$_{2^m}$, the truth of this variant seems to depend on the chosen universal $n$-randomness test. However, since our main point was to illustrate that layerwise reducibility is weaker than $n$-reducibility to $n$-randomness, we will not pursue this topic further.
\end{rem}

Next, let us study the result from Kautz \cite{kautz-1991} which states that every 2-random degree is hyperimmune, i.e.\ that every 2-random set computes a function which is not computably dominated. Note that the non-computably-dominated functions form a $\Pi^0_3$-class, so Proposition \ref{prop-max-uni} tells us that the uniformity is at most $3$. However, our version of the effective 0-1-law (Theorem \ref{eff-0-1-uni}) tells us that it is even at most $2$. We next show that the uniformity is exactly 2.

\begin{thm}\label{thm-random-hi}
The uniformity of the non-computably-dominated functions to the 2-random sets is 2.
\end{thm}
\begin{proof}
The upper bound follows from the standard proof and our effective 0-1-law, as discussed above. Towards a contradiction, assume there exists a sequence $\V_0,\V_1,\dots \subseteq 2^\omega$ of uniformly $\Pi^0_1$-classes and uniformly computable $e_0,e_1,\dots$ which witness that the non-computably-dominated functions 1-reduce to the 2-random sets.
 Then at least one $\V_i$ has positive measure; without loss of generality we may assume this is $\V_0$. Let $q > 0$ be a rational such that $\V_0$ has measure at least $q$. We will use a majority vote argument to show that there is a computable function which is not computably dominated, a contradiction.
 
We will define a computable function $f$ and a sequence $\V_0 \supseteq \W_0 \supseteq \W_1 \supseteq \dots$ of $\Pi^0_1$-classes such that every $\W_i$ has measure at least $q \left(\frac{1}{2} + 2^{-i-2}\right)$ and such that $f(i) \geq \Phi_{e_0}(X)(i)$ for every $X \in \W_i$. In particular, $\W = \bigcap_{i \in \omega} \W_i$ has positive measure, so it is non-empty. Furthermore, if $X \in \W \subseteq \V_0$, then $f$ dominates the function $\Phi_{e_0}(X)$, a contradiction.

We let $W_{-1} = \V_0$. At stage $s$ we define $f(s)$ and $\W_s$. Let $U_{s-1}$ be a prefix-free set of strings such that $\llbracket U_{s-1} \rrbracket = \overline{\W_{s-1}}$. Look for $n \in \omega$ such that for at least measure $1 - q 2^{-s-3}$ many strings $\sigma \in 2^n$ we have that either $\sigma \in U_{s-1}$ or  $\Phi_{e_0}(\sigma)(s){\downarrow}$. Such an $n$ must exist since $\W_{s-1} \subseteq \V_0$. Once this happens, we let
\[\W_s = \W_{s-1} \setminus \left\llbracket \left\{\sigma \in 2^n \mid \Phi_{e_0}(\sigma)(s){\uparrow}\right\}\right\rrbracket\]
and we let $f(s)$ be the maximum of $\Phi_{e_0}(\sigma)(s)$ for those $\sigma \in 2^n$ for which $\Phi_{e_0}(\sigma)(s){\downarrow}$. Then $\W_s$ has measure at least
\[\mu(\W_{s-1}) - q 2^{-s-3} = q \left(\frac{1}{2} + 2^{-s-3} + 2^{-s-3}\right) = q \left(\frac{1}{2} + 2^{-s-2}\right),\]
and it is clear that $f(s) \geq \Phi_{e_0}(X)(s)$ for every $X \in \W_s$.
\end{proof}

Next, we consider another result from Kautz \cite{kautz-1991} which states that the 2-random sets even compute 1-generic sets.

\begin{thm}
The uniformity of 1-genericity to 2-randomness is 2.
\end{thm}
\begin{proof}
Again, that the uniformity is at most 2 follows from the standard proof
%(see e.g.\ \cite[Theorem 8.21.4]{downey-hirschfeldt-2010})
together with our effective 0-1-law (Theorem \ref{eff-0-1-uni}). If it were the case that 1-genericity 1-reduced to 2-randomness, then the non-computably-dominated functions would also 1-reduce to the 2-random sets, because the 1-generic sets uniformly compute functions which are not computably dominated (this follows directly from the proof of Kurtz \cite{kurtz-1983} that every (weakly) 1-generic set is hyperimmune).
This contradicts Theorem \ref{thm-random-hi}.
\end{proof}

%Note that these previous two ARE true layerwise, by the standard proof we only need an upper bound on the randomness deficiency

%\begin{prop}
%There is a non-computable c.e.\ set $A$ such that $\{A\} \ured{1} 1\textrm{-Random} \cap \Delta^0_2$.
%\end{prop}
%Again, this is not true layerwise, because in that case we could show that $A$ is computable in a similar way as we did above for DNC, a contradiction.

\section{Theory of the $1$-uniform degrees as a Brouwer algebra}\label{sec-brouwer-1}

In this section we study the propositional theory of $\M_1$ as a Brouwer algebra. For more background on this, we refer to Sorbi \cite{sorbi-1996}.

As for the Medvedev lattice, the theory of $\M_1$ is not IPC because the top element is join-irreducible. We want to study what happens if we look at factors $\M_1 / \A$, i.e.\ the quotient of $\M_1$ by the principal filter generated by $\A$. Skvortsova \cite{skvortsova-1988} has shown that there is such a factor of the Medvedev lattice for which the theory is IPC. Her techniques were slightly improved in Kuyper \cite{kuyper-2014}. We will show that these techniques can be adapted to prove that there is a factor of $\M_1$ which has as theory IPC (which implies that the theory of $\M_1$ is Jankov's logic $\mathrm{Jan}$, i.e.\ IPC plus the weak law of the excluded middle $\neg p \vee \neg\neg p$, see the proof of \cite[Corollary 5.3]{kuyper-2014}).

The crucial ingredient we need to adapt Skvortsova's techniques to $\M_1$ is the next proposition.

\begin{prop}
Let $\A \in \M_1$ be a Muchnik degree, i.e.\ the degree of some set of functions which is upwards closed under Turing reducibility.
Then $\A$ is meet-irreducible. In fact, for all $\B,\C \in \M_1$ we have
\[\A \to_1 (\B \otimes \C) \ugeq{1} (\A \to_\M \B) \otimes (\A \to_\M \C) \geq_\M (\A \to_1 \B) \otimes (\A \to_1 \C).\]
\end{prop}
\begin{proof}
Let $\A,\B,\C \in \M_1$ with $\A$ a Muchnik degree. If $\A$ is the degree of the empty mass problem, the result is certainly true; so, we may assume this is not the case.
It is not hard to see that the meet-irreducibility follows from the second claim. That
\[(\A \to_\M \B) \otimes (\A \to_\M \C) \geq_\M (\A \to_1 \B) \otimes (\A \to_1 \C)\]
is also not hard to see; it follows from the fact that a uniform reduction is a special case of an $n$-reduction.
So, it remains to prove that 
\[\A \to_1 (\B \otimes \C) \ugeq{1} (\A \to_\M \B) \otimes (\A \to_\M \C).\]
Let $u \conc e \conc f \in \A \to (\B \otimes \C)$. We claim: there exists a $\sigma \in \omega^{<\omega}$ and an $i \in \omega$ such that for every $h$ extending $\sigma$ we have $(h \oplus f) \oplus f \in \P^1_{\{u\}(i)}$, and such that
\[\Phi_{\{e\}(i)}((\sigma \oplus (f \restriction |\sigma|)) \oplus (f \restriction 2|\sigma|))(0){\downarrow}.\]

For now, let us assume this claim holds and explain how this proves the result. Namely, consider the $\Pi^0_1$-classes $\W_{u,e,\sigma,i}$ where all elements of $\W_{u,e,\sigma,i}$ are of the form $u \conc e \conc f$, and $u \conc e \conc f \in \W_{u,e,\sigma,i}$ if and only if $\{e\}(i)[|\sigma|]{\downarrow}$, $\{u\}(i)[|\sigma|]{\downarrow}$,
\[\Phi_{\{e\}(i)}((\sigma \oplus (f \restriction |\sigma|)) \oplus (f \restriction 2|\sigma|))(0){\downarrow},\]
and $(h \oplus f) \oplus f \in \P^1_{\{u\}(i)}$ for every $h$ extending $\sigma$. Then, if $u \conc e \conc f$ is in both $\A \to (\B \otimes \C)$ and $\W_{u,e,\sigma,i}$, compute
\[a = \Phi_{\{e\}(i)}((\sigma \oplus (f \restriction |\sigma|)) \oplus (f \restriction 2|\sigma|))(0).\]
Let $b$ be an index for the functional sending $h_0 \oplus h_1$ to $\Phi_{\{e\}(i)}(((\sigma \conc h_0) \oplus h_1) \oplus h_1)$ without the first bit. Then it can be directly verified that $a \conc b \conc f \in (\A \to_\M \B) \otimes (\A \to_\M \C)$, where we use that $\A$ is dense because it is upwards closed under Turing reducibility. Furthermore, the $\W_{u,e,\sigma,i}$ cover $\A \to_1 (\B \otimes \C)$ by the claim. Thus, this proves the result.

\bigskip

So, let us prove the claim. Towards a contradiction, assume such $\sigma$ and $i$ do not exist. Fix a $g \in \A$. Then we also know that $\sigma \conc g \in \A$ for every string $\sigma$, again since $\A$ is dense.

We will now construct a $h$ such that $h \oplus f \geq_T g$ (and hence $h \oplus f \in \A$) and such that $(h \oplus f) \oplus f \not\in \bigcup_{i \in \omega} \P^1_{\{e\}(i)}$, a contradiction. We construct $h = \bigcup_{i} \sigma_i$ by a finite extension argument. Let $\sigma_{-1} = \emptyset$. Given $\sigma_{i-1}$, to define $\sigma_i$ we let $\tau \supseteq \sigma_{i-1}$ be the first string such $\llbracket (\tau \oplus (f \restriction |\tau|)) \oplus (f \restriction 2|\tau|) \rrbracket \cap \P^1_{\{u\}(i)} = \emptyset$. Let $\sigma_i = \tau \conc g(i)$.

Note that such a string $\tau$ always exists: namely, if such a string did not exist, we have that $((\sigma_{i-1} \conc g) \oplus f) \oplus f \in \P^1_{\{u\}(i)}$. So, there is some $s \in \omega$ with $\Phi_{\{e\}(i)}((((\sigma_{i-1} \conc g) \oplus f) \oplus f)\restriction 4s)(0){\downarrow}$. Since we assumed the claim is false, we then know that there is some $v \in \omega^\omega$ extending $(\sigma_{i-1} \conc g)\restriction s$ for which $(v \oplus f) \oplus f \not\in \P^1_{\{u\}(i)}$. Thus, there is some $(\sigma_{i-1} \conc g)\restriction s \subseteq \tau \subseteq v$ such that $\llbracket (\tau \oplus (f \restriction |\tau|)) \oplus (f \restriction 2|\tau|) \rrbracket$ is disjoint from $\P^1_{\{u\}(i)}$, as desired.

Note that the entire procedure is $f$-computable, so $h \oplus f$ computes $g$. Furthermore, $(h \oplus f) \oplus f \not\in \bigcup_{i \in \omega} \V_{\{e\}(i)}$ by construction, which proves the claim and hence the result.
\end{proof}

In the terminology of \cite{skvortsova-1988}, we have just shown that the set of Muchnik degrees is \emph{canonical} in $\M_1$. This allows us to show the following result.

\begin{thm}
Let $A$ be a computably independent set, and let
\[\A = \{i \conc f \mid f \geq_T A^{[i]}\}.\]
Then $\Th(\M_1 / \A) = \mathrm{IPC}$.
\end{thm}
\begin{proof}
We give a sketch.
The proof uses the exact same technique as the proof of \cite[Theorem 1.1]{kuyper-2014}. The only real modification is that we need a new proof of the fact that the Muchnik degrees are canonical in $\M_1$, which we have just given.

In a bit more detail, the proof in \cite{kuyper-2014} proceeds as follows:
\begin{itemize}
\item First, \cite[Theorem 3.3]{kuyper-2014} shows that there are certain embeddings of $(\mathcal{P}(I),\supseteq)$ into intervals $\left[\B,\overline{\A}\right]_\M$. It can be directly verified that the proof given there also works for $\M_1$ (in fact, it works for every $\M_n$).
\item Next, \cite[Corollary 4.3]{kuyper-2014} and \cite[Corollary 4.5]{kuyper-2014} extend these embeddings to free Brouwer algebras. This uses the previous fact together with the fact that the set of Muchnik degrees is canonical in the Medvedev lattice. As we have just shown this is also true in $\M_1$; so these two corollaries also hold in $\M_1$.
\item Finally, \cite[Theorem 1.1]{kuyper-2014} combines these facts with some general lattice-theoretic facts and with some computability-theoretic facts. In the proof certain sets are constructed such that certain equalities hold between certain mass problems in the Medvedev lattice; but of course, if things are Medvedev-equivalent they are certainly 1-equivalent. Therefore the proof proceeds in the same way for $\M_1$.\qedhere
\end{itemize}
\end{proof}

\begin{cor}
\[\mathrm{Th}(\M_1) = \mathrm{Jan}.\]
\end{cor}
\begin{proof}
See the proof of \cite[Corollary 5.3]{kuyper-2014}.
\end{proof}

Let us next note that the technique discussed in this section does not work for $n \geq 2$.

\begin{prop}\label{prop-meet-n}
Let $n \in \omega$. If $f,g$ are $\Delta^0_n$, then $C(\{f\}) \otimes C(\{g\}) \equiv_n C(\{f,g\})$.
\end{prop}
\begin{proof}
If $n \leq 1$, this follows from the fact that the bottom element of $\M_n$ is meet-irreducible. So, assume $n \geq 2$.
Clearly $C(\{f\}) \otimes C(\{g\}) \geq_n C(\{f,g\})$; in fact this reduction is even a Medvedev reduction since it is just inclusion. For the converse, note that the upper cone of a $\Delta^0_n$-set is $\Sigma^0_{n+1}$. By Remark \ref{remark-sigma} above we therefore see that $C(\{f\}) \otimes C(\{g\}) \leq_n C(\{f,g\})$ by sending $h \in C(\{f,g\}) \cap C(\{f\})$ to $0 \conc h$ and $h \in C(\{f,g\}) \cap C(\{g\})$ to $0 \conc h$.
\end{proof}

\begin{cor}
The set of Muchnik degrees is not canonical in $\M_n$ for $n \geq 2$.
\end{cor}
\begin{proof}
Let $f,g$ be two incomparable $\Delta^0_2$-functions. Then $\B \otimes \C \leq_2 \A$ by the previous proposition.
On the other hand, we do not have $\B \leq_w \A$ nor $\C \leq_w \A$ since $f$ and $g$ are incomparable.
\end{proof}

Thus, if we want to study the theory of $\M_n$ for $n \geq 2$, a different technique is needed.

\section{Theory of the $n$-uniform degrees as a Brouwer algebra for $n \geq 4$}\label{sec-brouwer-2}

In Sorbi and Terwijn \cite{sorbi-terwijn-2012}, it is shown that there are factors of the Muchnik lattice which yield IPC. A different proof is given in Kuyper \cite{kuyper-2013-2}. By studying how uniform the reductions in that proof are, we show that such factors in fact exist for $\M_n$ with $n \geq 4$.

\begin{thm}
Let $\A$ be the class of 1-generic $\Delta^0_2$-functions together with the computable functions and let $n \geq 4$. Then $\M_n / \overline{\A}$ has theory IPC.
\end{thm}
\begin{proof}
This follows by a careful analysis of the proof for the Muchnik lattice in \cite{kuyper-2013-2}. In that proof, a function $\alpha$ from $\A$ to $2^{<\omega}$ satisfying certain properties is constructed. First note that we can see such a function as a partial function $\beta$ from $\omega \to 2^{<\omega}$ instead, identifying functions in $\A$ with their $\Delta^0_2$-indices. Then the graph of $\beta$ is $\Sigma^0_4$:
\begin{itemize}
\item Checking that $e$ is an index for a $\Delta^0_2$-set, i.e. that $\{e\}^{\emptyset'}$ is total, is $\Pi^0_3$.
\item Checking that $\{e\}^{\emptyset'}$ is 1-generic is $\Pi^0_3$: let $W_0,W_1,\dots$ be an effective enumeration of the c.e.\ sets, then $\{e\}^{\emptyset'}$ is 1-generic if and only if
\[\forall e \exists n \exists s (\forall t \geq s (\{e\}^{\emptyset'[t]} \restriction n \in \W_e[s]) \vee
\forall \tau \forall t \geq s (\tau \supseteq \{e\}^{\emptyset'[t]} \to \tau \not\in \W_e[t])).\]
\item In the construction we need to check if $\{e_0\}^{\emptyset'} \leq_T \{e_1\}^{\emptyset'}$ for $e_0,e_1$ which are indices for $\Delta^0_2$-sets, this is $\Sigma^0_4$:
\[\exists a \forall n \exists s \forall t \geq s \left(\{a\}^{\{e_1\}^{(\emptyset' \restriction s)[t]}}(n)[t]\downarrow=\{e_0\}^{\emptyset'[t]}(n)[t]\right).\]
\item In the construction we need to, given $e_0$ and finitely many points already defined, do some kind of splitting to find an $e_1$ satisfying certain properties (this happens in \cite[Theorem 4.3]{kuyper-2013-2}). We can find this index $e_1$ effectively from $e_0$ and the points already defined.
\end{itemize}

%Note that we do not need to check equality, because the fact that it preserves the Turing ordering automatically implies that indices for identical functions get sent to identical elements.

Now, any function $f$ is in some $\alpha^{-1}(C(\sigma))$ if and only if $f$ is not in $\A$ (which is $\Pi^0_4$) or if there exists an $e$ with $f = \{e\}^{\emptyset'}$ (which is $\Pi^0_3$)
such that $\beta(e)$ extends $\sigma$ (which is $\Sigma^0_4$, as argued above).
Thus, every $\alpha^{-1}(\sigma)$ is $\Sigma^0_5$. Using Remark \ref{remark-sigma} it is not hard to see that the meet of two $\Sigma^0_{n+1}$-classes is their union in $\M_n$. Furthermore, since each $\alpha^{-1}(C(\sigma))$ is upwards closed, and for upwards closed $\A$ and $\B$ we have that their join is just their intersection and that
\[\A \to_n \B = \{f \in \omega^\omega \mid \forall g \in \A (f \oplus g \in \B)\}\]
for all $n \in \omega+1$, we see that the Muchnik degrees of
$\{\alpha^{-1}(\sigma) \mid \sigma \in 2^{<\omega}\}$ and the $\M_n$-degrees of $\{\alpha^{-1}(\sigma) \mid \sigma \in 2^{<\omega}\}$ for $n \geq 4$ are all pairwise isomorphic. Since these are the only degrees used in the proof in \cite{kuyper-2013-2}, the remainder of the proof is now exactly the same as for the Muchnik lattice.
\end{proof}

\begin{cor}
For $n \geq 4$ we have that $\mathrm{Th}(\M_n) = \mathrm{Jan}$.
\end{cor}
\begin{proof}
See the proof of \cite[Corollary 5.3]{kuyper-2014}.
\end{proof}

Thus, we have dealt with the cases $n \leq 1$ and $n \geq 4$. The author currently does not know how to deal with the cases $n=2$ and $n=3$. However, we conjecture the following.

\begin{conj}
The propositional theories of $\M_2$ and $\M_3$ are $\mathrm{Jan}$; in fact, there are principal factors of $\M_2$ and $\M_3$ which have as propositional theory $\mathrm{IPC}$.
\end{conj}

\section{Elementary equivalence}\label{sec-elementary}

Finally, in this section we will show that the first-order theories of the $n$-degrees as lattices (or equivalently, as partially ordered sets) are pairwise different. Recall that a \emph{degree of solvability} is a degree of a singleton $\{f\}$. First, we show that Dyment's (Dyment is the maiden name of Skvortsova) definition of the degrees of solvability in the Medvedev lattice from \cite{dyment-1976} works in every $\M_n$.

\begin{prop}\label{prop-solv-def}
For every $n \in \omega +1$, the degrees of solvability in $\M_n$ are definable as exactly those $\A$ for which there is a $\B >_n \A$ such that every $\C >_n \A$ satisfies $\C \geq_n \B$.
\end{prop}
\begin{proof}
For one direction, given a degree of solvability $\{f\}$, let $\B = \{e \conc g \mid \Phi_e(g) = f \wedge g \not\leq_T f\}$. Then $\B >_n \A$ for every $n \in \omega+1$, and if $\C >_n \A$ is witnessed by $\V_0,\V_1,\dots$ and $e_0,e_1,\dots$, then we can witness $\C \geq_n \B$ by sending $g \in \V_i \cap \C$ to $e_i \conc g$.

For the converse, let us first consider the case $n = \omega$. Let $\A$ and $\B$ be as in the statement of the proposition. Then $\B \not\leq_w \A$ so there is an $f \in \A$ such that, letting $\C = \{f\}$, we have $\C \not\geq_w \B$. Furthermore, we clearly have $\C \geq_w \A$. Now, if $\C > \A$ we would have $\C \geq_w \B$ by our assumption on $\B$, a contradiction. Thus $\A \equiv_w \C$ and $\A$ is therefore a degree of solvability.

Finally, let us consider the converse direction for $n \in \omega$, in which case we follow the proof of \cite{dyment-1976} (see e.g.\ \cite[Theorem 2.3]{sorbi-1996}). That is, given $\A$ which is not a degree of solvability and $\B \not\leq_n \A$ we construct a $\C >_n \A$ with $\C \not\geq_n \B$. We construct $\C$ as a set of the form $\{x_i \conc f_i \mid i \in \omega\}$ with $f_i \in \A$, which ensures that $\C \geq_\M \A$ and hence $\C \geq_n \A$. Our remaining requirements are therefore
\begin{itemize}
\item $P_{e,u}$: $\exists j \in \omega. \{e\}(j){\uparrow} \vee \{u\}(j){\uparrow} \vee \Phi_{\{e\}(j)}(\C \cap \P^n_{\{u\}(j)}) \not\subseteq \B$
\item $R_{e,u}$: $\exists j \in \omega. \{e\}(j){\uparrow} \vee \{u\}(j){\uparrow} \vee \Phi_{\{e\}(j)}(\A \cap \P^n_{\{u\}(j)}) \not\subseteq \C$.
\end{itemize}

The construction is now as in the Medvedev case, so we refer to \cite[Theorem 2.3]{sorbi-1996} for the details.
\end{proof}

For the Medvedev lattice, the following result essentially appears in \cite[Theorem 3.5]{dyment-1976}, except that it is not phrased in terms of definability.

\begin{prop}\label{prop-muchnik-def}
For every $n \in \omega$, the Muchnik degree $\B$ of $\A$ (i.e.\ the degree of $C(\A) = \{f \in \omega^\omega \mid \exists g \in \A (f \geq_T g)\}$) is definable by the formula $\phi(\A,\B)$ given by
\[\forall \{f\} (\{f\} \geq_n \B \to \{f\} \geq_n \A) \wedge \forall \C (\forall \{f\} (\{f\} \geq_n \C \to \{f\} \geq_n \A) \to \C \geq_n \B),\]
where we use the result from Proposition \ref{prop-solv-def} that the degrees of solvability are definable.
In particular, Muchnik reducibility is definable by
\[\psi(\A_0,\A_1) = \forall B_0,B_1 (\phi(\A_0,\B_0) \wedge \phi(\A_1,\B_1) \to \B_0 \leq_n \B_1).\]
\end{prop}
\begin{proof}
First, given $\A$, let $\B$ be the Muchnik degree of $\A$. Then clearly $\{f\} \geq_n \B$ implies $\{f\} \geq_n \A$; in fact, $\{f\} \geq_w \B$ implies $\{f\} \geq_\M \B$. Next, let $\C$ be such that for every $f$ with $\{f\} \geq_n \C$ we have $\{f\} \geq_n \A$. Then in particular for every $f \in \C$ we have $f \geq_T g$ for some $g \in \A$, and therefore $f \in \B$. Thus $\phi(\A,\B)$ holds.

Conversely, given $\A,\B$ with $\phi(\A,\B)$, let $\C$ be the Muchnik degree of $\A$. We will show that $\B \equiv_n \C$. First, if $f \in \B$ then $\{f\} \geq_n \B$, so $\{f\} \geq_n \A$ by the first conjunct of $\phi(\A,\B)$. So, $f \geq_T g$ for some $g \in \A$, and therefore we see $\B \subseteq \C$; thus $\C \leq_\M \B$. Finally, apply the second conjunct of $\phi(\A,\B)$ to $\C$ to obtain $\C \geq_n \B$.
\end{proof}
%Merk op: formule hier en bij vorige propositie onafhankelijk van n

In Proposition \ref{prop-meet-n} we gave a sufficient condition so that the meet of two Muchnik degrees $C(\{f\})$ and $C(\{g\})$ is $n$-equivalent to $C(\{f,g\})$. We now give an example of a situation in which this is not the case.

\begin{prop}\label{prop-meet-n-not}
Let $n \in \omega$ with $n \geq 1$, let $X$ be weak $n$-random, let $Y$ be $n$-random and let $X$ and $Y$ be Turing incomparable (in particular, these conditions are all satisfied if $X \oplus Y$ is $n$-random). Then $C(\{X\}) \otimes C(\{Y\}) \not\leq_n C(\{X,Y\})$.
\end{prop}
\begin{proof}
Assume towards a contradiction that $C(\{X\}) \otimes C(\{Y\}) \leq_n C(\{X,Y\})$ is witnessed by $\V_0,\V_1,\dots$ and $e_0,e_1,\dots$. Since the $\V_i$ cover $C(\{X\})$ we know there is some $i \in \omega$ with $X \in \V_i$. Fix such an $i$. Determine $s \in \omega$ such that $\Phi_{e_i}(X)[s]{\downarrow}$. Note that then $\Phi_{e_i}(X)[s] = 0$ because $X$ does not compute $Y$. Also, since $X$ is weakly $n$-random we know that $\mu(V_i \cap \llbracket X \restriction s \rrbracket) > 0$, because $X$ is not in any $\Pi^0_n$-class of measure $0$. So, by the effective 0-1-law (Theorem \ref{eff-0-1}) applied to $Y$ we know that $Y \restriction [k,\infty) \in V_i \cap \llbracket X \restriction s \rrbracket$ for some $k \in \omega$. But then $X \restriction s \subseteq Y$ so $\Phi_{e_i}(Y)(0){\downarrow}=0$, so $\Phi_{e_i}(Y) {\downarrow} \in 0 \conc C(\{X\})$. Thus $Y \geq_T X$, a contradiction.
\end{proof}

We now show that $\M_n$ and $\M_m$ are not elementarily equivalent for almost all $n \not= m$; we have to exclude the case $n=0$ and $m=1$.

\begin{thm}\label{thm-el-eq}
Let $n,m \in \omega + 1$ with $n < m$ and $m \geq 2$. Then $\M_n$ and $\M_m$ are not elementarily equivalent.
\end{thm}
\begin{proof}
First, if $m = \omega$ this follows from the fact that $\leq_n$ and $\leq_w$ do not coincide, as shown in Corollary \ref{cor-proper}, together with the fact shown in Proposition \ref{prop-muchnik-def} above that Muchnik reducibility is definable in $\M_n$. That is, take the formula $\phi$ which says that there are $\A$ and $\B$ such that $\A \not\leq \B$ while $\A \leq_w \B$, then $\phi$ holds with $\leq$ interpreted as $\leq_n$, but clearly not with $\leq$ interpreted as $\leq_w$.

Next, let $m \in \omega$. By Shore and Slaman \cite{shore-slaman-1999} we know that the jump is definable in the Turing degrees, so in particular the $\Delta^0_n$-degrees are definable in the Turing degrees.
Since the degrees of solvability are definable, as shown in Proposition \ref{prop-solv-def} above, we can therefore express the statement ``for all $f,g \in \Delta^0_m$ we have that $C(\{f\}) \otimes C(\{g\})$ is a Muchnik degree'' by a first-order formula $\phi$. Then $\phi$ holds in $\M_m$ by Proposition \ref{prop-meet-n}. On the other hand, since the Muchnik degree of $C(\{f\}) \otimes C(\{g\})$ is given by $C(\{f,g\})$, we see that $\phi$ does not hold in $\M_n$ by Proposition \ref{prop-meet-n-not}.
\end{proof}

We should note that the result of Shore and Slaman on the definability of the jump used above is very complex, and that it is probably too strong a tool for the simple thing we wish to prove. However, the author currently does not know of an easier example separating the first-order theories of these lattices.

We conclude with the following open question concerning the single case excluded in Theorem \ref{thm-el-eq}.

\begin{question}
Are $\M_0$ and $\M_1$ elementarily equivalent?
\end{question}

\section*{Acknowledgements}

The author wishes to acknowledge a helpful discussion with Takayuki Kihara while they were both visiting Andr\'e Nies at the University of Auckland. This discussion is what introduced the author to the work of Higuchi and Kihara \cite{higuchi-kihara-2014,higuchi-kihara-2014-2}, which was the main motivation for the author to study the topics discussed in this paper.

\newcommand{\noopsort}[1]{}
\providecommand{\bysame}{\leavevmode\hbox to3em{\hrulefill}\thinspace}
\providecommand{\MR}{\relax\ifhmode\unskip\space\fi MR }
\providecommand{\MRhref}[2]{%
  \href{http://www.ams.org/mathscinet-getitem?mr=#1}{#2}
}
\providecommand{\href}[2]{#2}

\end{document}